\documentclass[11pt]{amsart}

 \usepackage[T1]{fontenc}
 \usepackage[utf8]{inputenc}
 \usepackage[english]{babel}
 \usepackage{graphicx}
 \usepackage{geometry}

 \usepackage{amsmath,amsthm,amsfonts,amssymb}
 \usepackage{enumerate}
 \usepackage{color}
 \usepackage{fancyhdr}
 \usepackage[titletoc]{appendix}
 \usepackage{hyperref}
 \usepackage{tikz}
 \usepackage{tikz-cd}
 \usepackage{caption}
 \usetikzlibrary{matrix,arrows,decorations.pathmorphing, positioning}

\newcommand\ce{{\hspace{2.4 pt} \searrow\hspace{-8 pt}^e \hspace{5 pt}}}
\newcommand\ee{{\hspace{3 pt} \nearrow\hspace{-13 pt}^e \hspace{8 pt}}}
\newcommand\co{{\hspace{2 pt}\searrow \hspace{3 pt}}}
\newcommand\ex{{\nearrow \hspace{3 pt}}}
\newcommand\esc{{\searrow\hspace{-6 pt}\searrow\hspace{-8 pt}^e \hspace{5 pt}}}
\newcommand\sco{{\searrow \hspace{-6 pt}\searrow\hspace{3 pt}}}
\newcommand\ese{{\hspace{3 pt} \nearrow\hspace{-16 pt}^e\nearrow \hspace{3 pt}}}
\newcommand\sex{{\nearrow \hspace{-6 pt}\nearrow\hspace{3 pt}}}

\newtheorem{theo}{Theorem}[section]

\newtheorem{lemma}[theo]{Lemma}
\newtheorem{prop}[theo]{Proposition}
\theoremstyle{definition}
\newtheorem{defi}[theo]{Definition}
\theoremstyle{definition}
\newtheorem{rem}[theo]{Remark}

\newtheorem{example}[theo]{Example}

\input{xy}
\xyoption{all}
\xyoption{poly}
\usepackage[all]{xy}

 \newcommand\ZZ{{\mathbb{Z}}}
 \newcommand\NN{{\mathbb{N}}}

 \newcommand\st{{\mathrm{st}}}
 \newcommand\lk{{\mathrm{lk}}}
 
 \DeclareMathOperator{\gcat}{gcat}
\DeclareMathOperator{\plgcat}{plgcat}
 \DeclareMathOperator{\rk}{rk}

\title{Some remarks on PL collapsible covers of 2-dimensional polyhedra}
   \author{Eugenio Borghini}
   \address{Departamento  de Matemática - IMAS\\
 FCEyN, Universidad de Buenos Aires. Buenos Aires, Argentina.}
\email{eborghini@dm.uba.ar}



\begin{document}

   \begin{abstract}
  We analyze the topology and geometry of a polyhedron of dimension 2 according to the minimum size of a cover by PL collapsible polyhedra. We provide partial characterizations of the polyhedra of dimension 2 that can be decomposed as the union of two PL collapsible subpolyhedra in terms of their simple homotopy type and certain local properties. In the process, a special class of polyhedra of dimension 2 appears naturally. We give a combinatorial description of the spaces in this class, which includes all closed surfaces and the complexes associated to one-relator presentations.
   \end{abstract}

   \maketitle

\section{Introduction}

The Lusternik-Schnirelmann (L-S) category of a topological space $X$ is the minimum cardinality of a cover of $X$ by open sets which are contractible in the space. It is a classical homotopy invariant of a space, introduced in \cite{LS}, which has became over the years an important tool in homotopy theory (see \cite{CLOT} for a good account on the subject). A natural upper bound for the L-S category of a space $X$ is provided by its \textit{geometric category} $\gcat(X)$, defined as the minimum number of open contractible sets that cover $X$. For a polyhedron $P$ (i.e. the underlying topological space of some simplicial complex), the geometric category coincides with the minimum number of contractible subpolyhedra that cover $P$.
\par In this note we propose to study a variant of the geometric category in the context of compact connected polyhedra, which we call \textit{PL geometric category} and denote it by $\plgcat$. For this invariant, we replace the purely topological notion of ``contractible'' in the definition of geometric category by the more geometrically flavored notion of ``PL collapsible'' (refer to Section \ref{definitions} for precise definitions). This point of view allows to exploit certain combinatorial properties of a space that admits triangulations while at the same time accounts for its inherent geometry and topology. We show in first place that the PL geometric category of a polyhedron of dimension $n$ is bounded by $n+1$, thus generalizing the corresponding result for geometric category (cf. \cite[Proposition 3.2]{CLOT}). This implies that the PL geometric category of a non PL collapsible polyhedron of dimension 2 may only be 2 or 3. One of our main objectives is to understand the topological and geometrical properties that distinguish 2-dimensional polyhedra $P$ with $\plgcat(P) = 2$ from those with $\plgcat(P) = 3$. In this direction, we find that the condition of having PL geometric category 2 is fairly restrictive. In particular, it determines the simple homotopy type of the polyhedron: by Proposition \ref{plgcat2} below, such a polyhedron is simple homotopy equivalent to a wedge sum of spheres of dimension 1 and 2. Moreover, it is not difficult to verify that a contractible polyhedron $P$ of dimension 2 with $\plgcat(P) = 2$ satisfies the Andrews-Curtis conjecture \cite{AC}, which states that a compact contractible 2-dimensional polyhedron $3$-deforms to a point (see Remark \ref{AC}). However, as observed in Section \ref{definitions}, the PL geometric category is not a (simple) homotopy invariant of a polyhedron. This leads us to study this invariant also from a local point of view. In this context we describe a special class of 2-dimensional polyhedra, which we call \textit{inner-connected polyhedra}, defined by a property satisfied among others by all closed surfaces. We obtain a criterion which states that sufficiently ``regular'' inner-connected polyhedra cannot have PL geometric category 2 (see Theorem \ref{No2Theo}).
\par Section \ref{onerel} is devoted to the computation of the PL geometric category of complexes associated to one-relator presentations. The main result of this section, Theorem \ref{plgcatonerel}, shows that it is possible to read off the PL geometric category of such a complex directly from the presentation.
\par In the final section of the article we investigate in more detail the class of inner-connected polyhedra. We prove that any such polyhedron is obtained from a polygon with some identifications performed on its sides. This is a generalization of the well-known result that closed surfaces admit a polygonal presentation with only one face (see for example \cite{Lee}).
\par In the recent works \cite{FMV}, \cite{AS} discrete versions of the L-S category and related invariants were introduced in the setting of finite simplicial complexes and finite topological spaces. These discrete versions depend mainly on the combinatorial structure of the involved spaces. Our notion relies more strongly on the topology and geometry of the underlying spaces.

\section{PL geometric category}\label{definitions}

In this section we introduce the notion of \textit{PL geometric category} of a polyhedron. We collect some necessary definitions first. 
\par By a \textit{polyhedron} we understand a topological space which admits triangulations, i.e. the underlying space of some simplicial complex. A subspace $Q$ of a polyhedron $P$ is a \textit{subpolyhedron} if it is the underlying space of a subcomplex of some triangulation of $P$.  We recall next the basic definitions of Whitehead's simple homotopy theory. Let $K$ be a finite simplicial complex. A simplex $\sigma$ of $K$ is a \textit{free face} of $K$ if there is a unique simplex $\tau \in K$ containing $\sigma$. In that case, we say that there is an \textit{elementary collapse} from $K$ to $L = K \setminus\{\sigma, \tau\}$, denoted $K \ce L$. More generally, $K$ \textit{collapses} to $L$, denoted by $K \co L$, if there is a sequence $K_1 = K, K_2, \dots, K_{r} = L$ such that $K_i \ce K_{i+1}$ for every $i$. We also say that $L$ expands to $K$ and denote $L \ex K$. The complex $K$ is called \textit{collapsible} if it collapses to a complex with only one vertex. A pair of simplicial complexes $K$ and $L$ are \textit{simple homotopy equivalent} if there exists a finite sequence of complexes $K_1 = K, K_2, \dots, K_r = L$ such that for every $i$ either $K_i \ee K_{i+1}$ or $K_i \ce K_{i+1}$. In that situation, we also say that there is an $n$-deformation from $K$ to $L$ if the dimension of complexes $K_1, \dots K_r$ is at most $n$. A polyhedron $P$ \textit{PL collapses} to a subpolyhedron $Q$ (and we still denote $P \co Q$) if there exist coherent triangulations $K$, $L$ of $P$ and $Q$ respectively such that $K \co L$ (see \cite[Ch.2]{Hud}). A polyhedron $P$ is called \textit{PL collapsible} if it PL collapses to a point, (i.e. some simplicial complex that triangulates $P$ collapses to a vertex).
\par The polyhedra that we work with are assumed to be compact and connected. Likewise, the simplicial complexes are assumed to be finite and connected.

\begin{defi}
Let $P$ be a polyhedron. The \textit{PL geometric category} $\plgcat(P)$ of $P$ is the minimum number of PL collapsible subpolyhedra that cover $P$.
\end{defi}

It is a well-known fact that the geometric category of a (compact, connected) polyhedron $P$ of dimension $n$ is at most $n+1$. 
We will show an analogous result for PL geometric category, namely, that a polyhedron of dimension $n$ is covered by at most $n+1$ PL collapsible subpolyhedra. The strategy for proving this, similarly as in the proof of the geometric category version, is to proceed by induction on the dimension of the polyhedron. However, for the inductive step to work in our context, a slight technical detour is needed. 
Specifically, we resort to the theory of strong homotopy types of \cite{BM12}.

\begin{defi} \cite{BM12}
Let $K$ be a simplicial complex and $v \in K$ a vertex. We say that $v$ is dominated by a vertex $v' \neq v$ if every maximal simplex that contains $v$ also contains $v'$. If $v$ is dominated by some vertex $v'$, we say that there is an \textit{elementary strong collapse} from $K$ to $K \setminus v$ and denote $K \esc K \setminus v$. In that situation we also say that there is an \textit{elementary strong expansion} from $L = K \setminus v$ to $K$ and denote it by $L \ese K$. If there is a sequence of elementary strong collapses that starts in $K$ and ends in $L$, we say that there is a \textit{strong collapse} from $K$ to $L$ and denote $K \sco L$.  The inverse of a strong collapse is called a \textit{strong expansion} and denoted by $L \sex K$.
\end{defi}

\begin{rem}
\label{strongrem} \cite[Remark 2.4]{BM12} $K \sco L$ implies that $K \co L$.
\end{rem}

Recall that the \textit{star} of a vertex $v$ in a simplicial complex $K$ is the subcomplex $\st_{K}(v) \subseteq K$ formed by the union of the simplices $\sigma \in K$ such that $\sigma \cup {v} \in K$. The \textit{link} of $v$ is the subcomplex $\lk_{K}(v) \subseteq \st_{K}(v)$ of the simplices that do not contain $v$. For a given simplex $\sigma$, its \textit{boundary} $\dot{\sigma}$ is the subcomplex formed by the simplices $\tau$ strictly contained in $\sigma$.

\begin{rem}
\label{simplicialCone} A vertex $v$ in a simplicial complex $K$ is dominated by $v'$ if and only if the link $\lk_{K}(v)$ is a simplicial cone with apex $v'$, i.e. $\lk_{K}(v) = v'M$ for certain subcomplex $M$.
\end{rem}

\begin{lemma}\label{strongcollapse}
Let $\sigma_n$ be the standard $n$-simplex. Consider the subcomplex of the second barycentric subdivision of $\sigma_n$ defined as $K_n := \sigma_n'' \setminus \st_{\sigma_{n}''}(\{v\})$, where $v$ is the barycenter of $\sigma_n$. Then $K_n$ strong collapses to $(\dot{\sigma}_n)''$.
\end{lemma}

\begin{proof}
We view the simplices of the second subdivision of $\sigma_n$ as chains of simplices of $\sigma_n'$ ordered by inclusion. Let $w \in \lk_{\sigma_n''}(\{v\})$ be a vertex. This means that $\{v\} \cup w$ forms a $1$-simplex in $\sigma_n''$ and so there is a chain of inclusion of simplices $\{ \{v\} \subseteq w \}$. If the length of the chain $w$ is two, say $w = \{ v \subseteq a \}$, any maximal simplex of $\sigma_n''$ containing $w$ either contains $v$ or $a$. Since $v \not\in K_n$, this shows that $w$ is dominated by $a$ in $K_n$. By removing the vertices of $\lk_{\sigma_n''}(\{v\})$ in non-decreasing order of the length of the chain they represent, we see that $K_n \sco (\dot{\sigma}_n)''$.
\end{proof}

\begin{lemma}\label{induc}
Let $K$, $L$ be simplicial complexes such that $L \sex K$. If $L$ can be covered by $n$ strong collapsible subcomplexes, so does $K$.
\end{lemma}

\begin{proof}
Let $\{L_1, \dots, L_n\}$ be a cover of $L$ by $n$ strong collapsible subcomplexes and assume that there is an elementary strong expansion from $L$ to $K$, say $L = K \setminus v$ for certain $v \in K$. Let $v' \in K$ be a vertex that dominates $v$, so that $\lk_{K}(v) = v'M$ for some subcomplex $M$ of $L$. For each $1 \leq i \leq n$, define the subcomplex $K_i$ of $K$ as
\[
K_i = 
\begin{cases} 
	L_i \cup v(v'M \cap L_i) \text{ if } v'M \cap L_i \neq \emptyset, \\ 
	L_i \text{ otherwise.}
\end{cases}
\]
If $v'M \cap L_i$ is nonempty, then $v \in K_i$ and is clearly dominated by $v'$ because $\lk_{K_i}(v) = v'(M \cap L_i)$. In any case, $K_i$ strong collapses to $L_i$ and is therefore strong collapsible. This shows that $K$ is covered by $n$ strong collapsible subcomplexes. The conclusion follows by induction on the number of elementary strong expansions from $L$ to $K$.
\end{proof}

We are now able to prove that the PL geometric category of a polyhedron of dimension $n$ is bounded from above by $n+1$. We will prove the following slightly stronger result.

\begin{prop}\label{plgcatDimN}
Let $K$ be a complex of dimension $n$. Then, the second barycentric subdivision $K''$ of $K$ can be covered by $n+1$ strong collapsible subcomplexes.
\end{prop}

\begin{proof}
Proceed by induction on $n$, the dimension of $K$. When $n = 1$, $K$ is a simplicial graph. We show first that in this case $K'$ admits a cover by two strong collapsible subcomplexes. In order to produce the strong collapsible cover, let $T$ be a spanning tree of the graph $K$ and note that each edge $e \in K \setminus T$ becomes the union of two edges in $K'$, say $e = e_1 \cup e_2$. Consider the following subcomplexes of $K'$:
\[ K_1 = T' \cup \bigcup_{e \in K \setminus T} e_1 \text{ , } K_2 = T' \cup \bigcup_{e \in K \setminus T} e_2. \]
As $K_1$, $K_2$ both strong collapse to $T'$, they are strong collapsible and they clearly cover $K'$. Since their barycentric subdivions are also strong collapsible, the base case is complete.
\par Let now $K$ be a simplicial complex of dimension $n$. By inductive hypothesis, the second barycentric subdivision of the ($n-1$)-skeleton $\left(K^{(n-1)}\right)''$ of $K$ can be covered by $n$ strong collapsible subcomplexes $K_1, \dots, K_n$. Let $v_1, \dots , v_r$ be the barycenters of the maximal simplices of $K$. By Lemma \ref{strongcollapse}, we see that $\left(K^{(n-1)}\right)'' \sco K'' \setminus \bigcup_{i=1}^{r} \st_{K''}(\{v_i\})$ and so Lemma \ref{induc} implies that this last complex is covered by $n$ strong collapsible subcomplexes. Since $K''$ is connected and $\st_{K''}(\{v_i\})$ is strong collapsible for every $i$, we can include their union in a strong collapsible subcomplex of $K''$.
\end{proof}

As a consequence, in dimension 1 the PL geometric category only distinguishes trees (contractible graphs) from the rest of graphs. The first non trivial case is the class of 2-dimensional polyhedra. Since PL collapsible polyhedra are relatively well understood, the interest is centered in understanding the difference between polyhedra of PL geometric category 2 from those of PL geometric category 3. Our first step in this direction concerns the simple homotopy type of a polyhedron $P$ of dimension 2 with $\plgcat(P) = 2$. By a result of C.T.C. Wall \cite{Wa}, a polyhedron $P$ with $\gcat(P) = 2$ has the homotopy type of a finite wedge sum of spheres of dimension $1$ and $2$. We show that a polyhedron $P$ with $\plgcat(P) = 2$ $3$-deforms to the suspension of a graph.

\begin{lemma}\label{suspension}
Let $K$ be a simplicial complex of dimension 2 which is covered by collapsible subcomplexes $K_1$, $K_2$. Then there is a $3$-deformation from $K$ to the suspension $\Sigma(K_1 \cap K_2)$ of $K_1 \cap K_2$.
\end{lemma}

\begin{proof}
Cone off $K_1$, $K_2$ with vertices $v_1$, $v_2$. This gives an expansion $K \ex v_1K_1 \cup v_2K_2$. Collapse every new simplex based on a simplex contained in $K_1$ or $K_2$ but not in both. Hence, $K \ex v_1K_1 \cup v_2K_2 \co v_1 (K_1 \cap K_2) \cup v_2 (K_1 \cap K_2)$, which is the desired 3-deformation.
\end{proof}

\begin{lemma}\label{SubofCollapsible}
Let $K$ be a collapsible simplicial complex of dimension $2$ and $L$ a subcomplex of $K$. If $\dim L = 2$, $L$ collapses to a graph, i.e. a complex of dimension 1.
\end{lemma}
\begin{proof}
Choose an ordering $\sigma_1, \sigma_2, \dots, \sigma_r$ of the $2$-simplices of $K$ that induces a valid sequence of collapses. It is clear then that the first $2$-simplex of $L$ appearing in that list must have a free face in $L$ and hence $L$ collapses to a subcomplex with one fewer 2-simplex. By induction on the number of $2$-simplices of $L$, it follows that $L$ collapses to a graph.
\end{proof}

\begin{prop}\label{plgcat2}
Let $P$ be a polyhedron of dimension $2$ such that $\plgcat(P) = 2$. Then $P$ $3$-deforms to the suspension of a graph.
\end{prop}
\begin{proof} Take a triangulation $K$ of $P$ covered by collapsible subcomplexes $K_1$, $K_2$. By Lemma \ref{suspension}, $K$ $3$-deforms to $\Sigma(K_1 \cap K_2) = v_1(K_1 \cap K_2) \cup v_2(K_1 \cap K_2)$ and by Lemma \ref{SubofCollapsible} $K_1 \cap K_2$ collapses to a $1$-dimensional subcomplex $G$. It follows that $v_i(K_1 \cap K_2) \co v_iG$ for $i=1,2$, and hence $K$ $3$-deforms to the suspension of $G$.
\end{proof}

\begin{rem}\label{AC} As a consequence of Proposition \ref{plgcat2}, the Andrews-Curtis conjecture is satisfied by contractible polyhedra which admit a cover by two PL collapsible subpolyhedra. Indeed, let $P$ be a contractible polyhedron covered by collapsible subpolyhedra $P_1$, $P_2$. From the Mayer Vietoris sequence, the intersection $P_1 \cap P_2$ has trivial homology and by Lemma \ref{SubofCollapsible}, $P_1 \cap P_2$ collapses to a tree. By Proposition \ref{plgcat2}, $P$ 3-deforms to a point. 
\end{rem}

As it was to be expected, the property of having PL geometric category 2 is not a (simple) homotopy invariant of a polyhedron. To illustrate this point, we invoke the classical example used by Fox \cite{Fox} to show that the geometric category is not a homotopy invariant. Let $P_1$ be the wedge sum of $S^2$ and two circles and let $P_2$ be the space obtained from $S^2$ by identifying three distinct points. Notice that $P_1$ and $P_2$ are simply homotopy equivalent (in fact, there is a $3$-deformation from $P_1$ to $P_2$). By splitting every sphere in $P_1$ in two, we see that $P_1$ admits a cover by two PL collapsible subpolyhedra and hence $\plgcat(P_1) = 2$. On the other hand, since $P_2$ does not admit covers by two contractible subpolyhedra by \cite[§39]{Fox}, $\plgcat(P_2) = 3$.
\par Thus, the global simple homotopy type is not enough to characterize 2-dimensional polyhedra of PL geometric category 2. A study of a more local nature is required. In this context a special class of polyhedra of dimension 2, which we proceed to describe, appears naturally.

\begin{defi}
Let $K$ be a simplicial complex of dimension $2$. We say that an edge of $K$ is \textit{inner} if it is a face of exactly two $2$-simplices of $K$.
\end{defi}

Recall that a simplicial complex $K$ of dimension $n$ is \textit{homogeneous} or \textit{pure} if all of its maximal simplices have dimension $n$.

\begin{defi}
Let $K$ be a homogeneous $2$-dimensional simplicial complex. We say that $K$ is \textit{inner-connected} if any pair of $2$-simplices $\sigma, \tau$ of $K$ is connected by a sequence of 2-simplices $\sigma = \eta_1, \eta_2, \dots, \eta_r = \tau$ such that $\eta_i \cap \eta_{i+1}$ is an inner edge of $K$ for each $1 \leq i < r$. We call such a sequence an \textit{inner sequence}. We say that a polyhedron $P$ is inner-connected if one (=all) of its triangulations is inner-connected.
\end{defi}

Recall that $K$ is strongly connected if it is homogeneous and for every pair of 2-simplices $\sigma, \tau$ of $K$ there is a sequence of $2$-simplices of $K$ $\sigma_1 = \sigma, \sigma_2, \dots, \sigma_n = \tau$ such that $\sigma_i \cap \sigma_{i+1}$ is an edge of $K$ for each $i$. Obviously, an inner-connected complex of dimension 2 is strongly connected. A strongly connected complex of dimension 2 is a \textit{pseudosurface} if each of its edges is a face of at most two 2-simplices .

\begin{example}
Surfaces or more generally pseudosurfaces are inner-connected. The presentation complex associated to finite one-relator presentation in which every generator appears at least once in the relator is also inner-connected.
\end{example}

Consider a 2-dimensional polyhedron $P$ which is the union of two collapsible subpolyhedra $P_1$, $P_2$. We know by Lemma \ref{SubofCollapsible} that the intersection $P_1 \cap P_2$ collapses to a graph. The main reason why inner-connected polyhedra are useful as a technical tool is the following: if $P$ is assumed to be inner-connected, it is possible to deform $P_1$ and $P_2$ so that $P_1 \cap P_2$ \textit{is} a graph.

\begin{lemma}\label{GraphInt}
Let $K$ be an inner-connected and non collapsible simplicial complex of dimension 2. Suppose that $K$ is the union of collapsible subcomplexes $K_1$, $K_2$. Then there exist collapsible subcomplexes $L_1$, $L_2$ such that $K = L_1 \cup L_2$ and $L_1 \cap L_2$ is 1-dimensional.
\end{lemma}

\begin{proof}
Suppose $K_1 \cap K_2$ has at least one $2$-simplex $\eta$. Since $K_1 \cap K_2$ is a proper subcomplex of $K$, we can find a $2$-simplex not in $K_1 \cap K_2$ and an inner sequence joining it to $\eta$. Then there are $2$-simplices $\sigma, \tau$ together with an inner edge $e = \sigma \cap \tau$ such that $\tau \in K_1 \cap K_2$ but $\sigma \not\in K_1 \cap K_2$. Without loss of generality, suppose $\sigma \in K_1$. Then $e$ is a free face of the complex $K_2$, which implies that we can remove $\tau$ from $K_2$. That is, the complexes $K_1$ and $\tilde{K}_2 = K_2 \setminus \tau$ form again a collapsible cover of $K$ and $K_1 \cap \tilde{K}_2$ has one fewer $2$-simplex than $K_1 \cap K_2$. It follows by induction that it is possible to find collapsible subcomplexes $L_1, L_2$ that cover $K$ and intersect in a graph.
\end{proof}

Even if a polyhedron $P$ admits covers by two PL collapsible subpolyhedra intersecting in a graph, the possible structure of these intersection graphs is constrained by the local topology of $P$. More concretely, we show that the topology of small neighborhoods around a point that is a leaf of an intersection graph satisfies certain condition.

\begin{defi} \cite{Hac}
Let $K$ be a simplicial complex. A vertex $v$ of $K$ is a \textit{bridge} if $K \setminus v$ has more connected components than $K$.
We say that $v$ is \textit{splittable} if the link $\lk_{K}(v)$ has bridges. Note that it makes sense to say that a point in a polyhedron is splittable because this property depends only on the homeomorphism type of a small closed neighborhood around the point and not on a specific triangulation of the space.
\end{defi}

The statement and proof of the following lemma are based on results from \cite{Bar11, Hac}.

\begin{lemma}\label{IrrVert}
Let $K$ be a homogeneous complex of dimension 2 which admits a collapsible cover of size two. Suppose additionally that the link of every non splittable vertex of $K$ is connected. Then, there exist collapsible subcomplexes $L_1$, $L_2$ that cover $K$ and such that every leaf of the 1-skeleton $(L_1 \cap L_2)^{(1)}$ of $L_1 \cap L_2$ is a splittable vertex of $K$.
\end{lemma}

\begin{proof}
Let $K_1$ and $K_2$ be subcomplexes of $K$ that form a collapsible cover of $K$. Take $\eta = vw \in (K_1 \cap K_2)^{(1)}$ an edge such that $w$ is a leaf, i.e. $\lk_{(K_1 \cap K_2)^{(1)}}(w) = v$, but not a splittable vertex. Suppose in first place that $\eta$ is not maximal in either of the subcomplexes $K_1, K_2$, so that there exist vertices $v_i \in K_i$ with $vwv_i \in K_i$ for $i = 1,2$. As $w$ is not a splittable vertex, we can find a path joining $v_1$ and $v_2$ in $\lk_K(w)\setminus v$. But then there must be at least another edge in $\lk_{K_1 \cap K_2}(w)$ contradicting the hypothesis that $\eta$ is a leaf of $(K_1 \cap K_2)^{(1)}$. Suppose now $\eta$ is maximal in $K_1$ and take $\tau = v_ 2\eta$ a $2$-simplex of $K_2$ containing $\eta$ (we can find one by homogeneity of $K$). We show that in this case $K_1$ collapses to $K_1 \setminus w$. If it was not the case, there should be another edge $\eta' \in K_1$ hanging from $w$. By the homogeneity of $K$, $\eta'$ is the face of some 2-simplex $\sigma = v_1 \eta'$ which per force is in $K_1$ but not in $K_2$. Since by hypothesis $w$ is not splittable and has connected link, there is a path in $\lk_K(w) \setminus v$ joining $v_1$ to $v_2$ and so $w$ cannot be a leaf of $C$, a contradiction. By performing the collapses that correspond to edges in the second case, we may assume the the leaves of $(K_1 \cap K_2)^{(1)}$ are splittable vertices.
\end{proof}

Consider again a 2-dimensional polyhedron $P$ covered by collapsible subpolyhedra $P_1$, $P_2$. A straightforward computation using the (reduced) Mayer-Vietoris long sequence reveals that
\[
 \tilde{H}_0(P_1 \cap P_2) \equiv H_1(P) \text{, }H_1(P_1 \cap P_2) \equiv H_2(P),
\]
where the homology groups are taken with coefficients in $\ZZ$. From Proposition \ref{plgcat2}, we know that $H_1(P)$ and $H_2(P)$ are finitely generated free abelian groups. Suppose that $\rk H_2(P) < \rk H_1(P)$. Since by Lemma \ref{SubofCollapsible} the polyhedron $P_1 \cap P_2$ collapses to a graph, at least two connected components of $P_1 \cap P_2$ are collapsible (because at least two of them are acyclic). When these components are graphs (for example, this is the case if $P$ is inner-connected), by Lemma \ref{IrrVert} its leaves should be located in splittable vertices or vertices with non connected links. Thus, $P$ should have at least two such vertices. The conclusion reached in this paragraph is roughly that an inner-connected polyhedron which is regular both in a local and a global sense does not admit PL collapsible covers of size two.

\begin{theo} \label{No2Theo}
Let $P$ be an inner-connected polyhedron of dimension $2$ such that $H_2(P) \equiv 0$ or $\rk H_2(P) < \rk H_1(P)$. Suppose additionally that $P$ is not PL collapsible, has at most one splittable vertex and that the link of every non splittable vertex is connected. Then $\plgcat(P) = 3$.
\end{theo}

\begin{proof}
The case $\rk H_2(P) < \rk H_1(P)$ was already treated in the paragraph above. Suppose then $H_2(P) \equiv H_1(P) \equiv 0$ and that $P$ is the union of PL collapsible subpolyhedra $P_1$, $P_2$ that intersect in a graph. Hence, $P_1 \cap P_2$ is a tree and since we may assume by Lemma \ref{IrrVert} that its leaves are located in splittable vertices, $P_1 \cap P_2$ should be a point. It follows that $P$ is a wedge sum of PL collapsible polyhedra, which contradicts the hypothesis that $P$ be inner-connected.
\end{proof}

\begin{example} The dunce hat $D$ is an inner-connected contractible polyhedron with only one splittable vertex and such that every other vertex has connected link. Hence, by Theorem \ref{No2Theo} no triangulation of $D$ admits a cover by two collapsible subcomplexes. In fact, we can say a little more. The dunce hat $D$ can be viewed as the presentation complex associated to the one-relator presentation $\langle \, a \,|\, aaa^{-1} \,\rangle$ (see the first paragraph of Section \ref{onerel}). More generally, by Theorem \ref{No2Theo} none of the presentation complexes associated to a presentation of the form $\langle \, a \,|\, a^{n}a^{-(n-1)} \,\rangle$ $(n \geq 2)$ admits a cover by two PL collapsible subpolyhedra.
\end{example}

\begin{example} The standard Bing's house with two rooms admits a PL collapsible cover of size two (to see this, split the complex in two halves, each one containing the walls which support the vertical tunnels). However, as a consequence of the proof of Theorem \ref{No2Theo} it is impossible to cover this polyhedron by two PL collapsible subpolyhedra intersecting in a graph.
\end{example}

\section{The geometry of one-relator presentations} \label{onerel}

We use the results of the previous section to provide a complete characterization of one-relator presentation complexes that admit a PL collapsible cover of size two. 
\par Recall that associated to a finite presentation ${\mathcal P} = \langle \, X \,|\, R \,\rangle $ there is a topological model built as follows. Let $K = \vee_{x \in X} S_{x}^{1}$ be a wedge sum of $1$-spheres indexed by $X$. Every word $r \in R$ spells out a combinatorial loop on the space $K$ based on the wedge point, which is used to attach a $2$-cell on $K$. The resulting $2$-dimensional \rm{CW}-complex is called the \textit{presentation complex} of ${\mathcal P}$ and is denoted by $K_{\mathcal P}$. Since the attaching maps are combinatorial, the presentation complex $K_{\mathcal P}$ is a polyhedron (see \cite[Chapter 2]{HAM} for more details). When the set $R$ consists of only one word $r$ the presentation $ \langle \, X \,|\, r \,\rangle $ is called a \textit{one-relator presentation}.
\par In what follows, we will assume that the one-relator presentation complexes are homogeneous, that is, every generator appears in the relator. There is no loss of generality in this assumption. Indeed, if it was not the case, the associated complex $K_{\mathcal P}$ would decompose as a wedge sum of a bouquet of 1-spheres and a homogeneous one-relator complex $K_{\mathcal Q}$. It is easy to see then that to compute $\plgcat(K_{\mathcal P})$, it is enough to compute $\plgcat(K_{\mathcal Q})$.

\begin{prop}\label{StripProp}
Let ${\mathcal P} = \langle \, x_1, \dots, x_k \,| \, r \, \rangle$ be a finite one-relator presentation and suppose that $r$ admits an algebraic collapse, that is, there is a generator $x$ which occurs only once in $r$ with exponent $\pm 1$. Then $K_{ {\mathcal P} }$ admits a cover by two PL collapsible subpolyhedra, that is, $\plgcat(K_{ {\mathcal P} }) \leq 2$.
\end{prop}
\begin{proof}
We may assume that $x = x_1$ and $r = x^{\pm 1} a_1 \dots a_{m-1}$, where each $a_i$ is equal to some $x_j^{\pm1}$, $j \neq 1$. Picture the complex $K_{\mathcal P}$ as a disk with the boundary subdivided in $m$ edges labeled in counterclockwise order according to $r$. Subdivide the edge labeled $x$ in $2(m-1) + 1$ edges and subdivide the rest of the edges in three edges. Join the $i$-th edge of the subdivided $x$ to the central edge of (the edge labeled as) $a_i$ by a 2-dimensional strip inside the disk in such a way that the strips are pairwise disjoint (see Figure \ref{stripfig}). 

\begin{figure}[!htpb]
\centering
\includegraphics[width=0.6\textwidth]{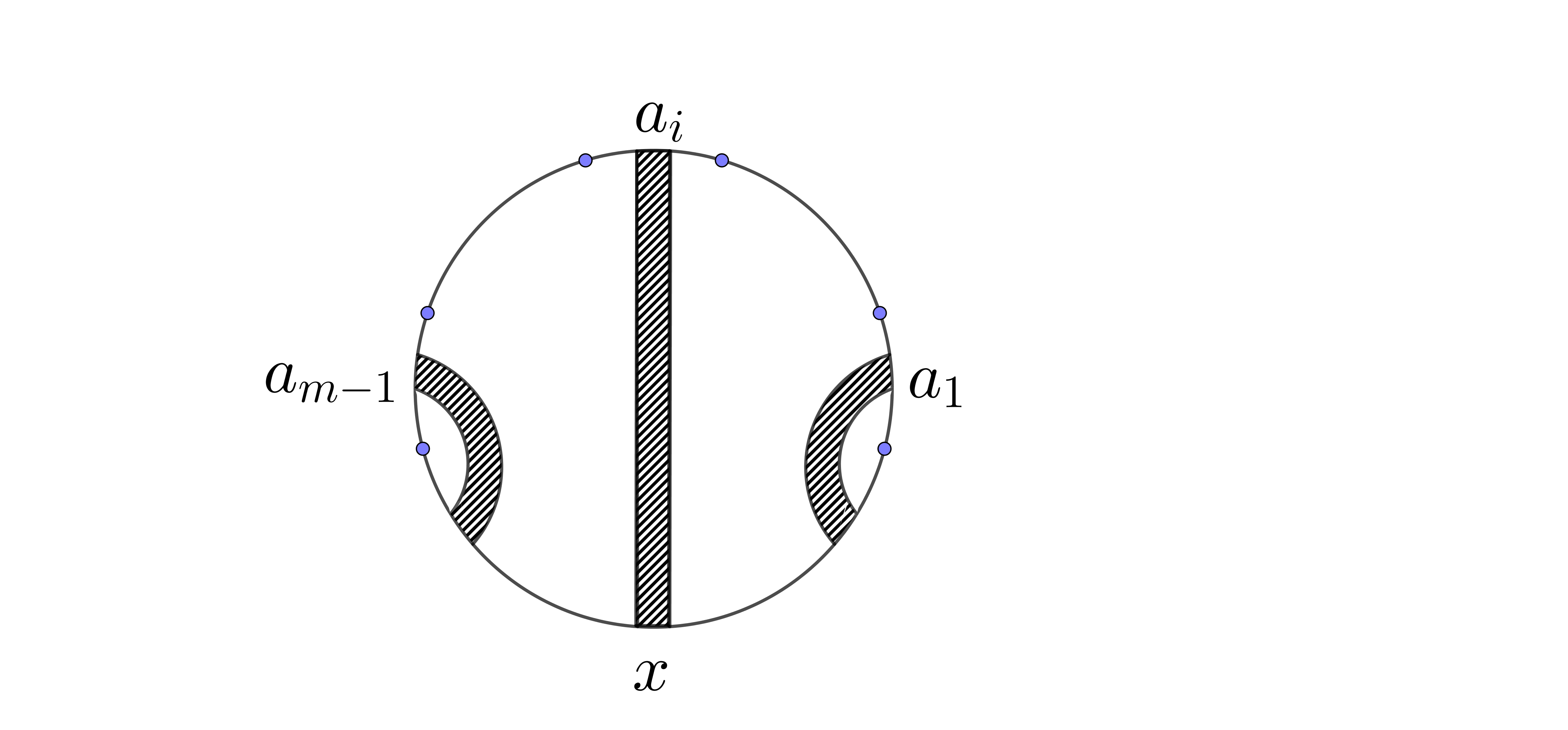}
\caption{The strips (shaded) PL collapse to a tree through the edge which intersects the edge labeled $x$.}\label{stripfig}
\end{figure}

Both the subpolyhedron $P_1$ formed by the union of these strips and its complement $P_2$ consist of a disjoint union of PL collapsible polyhedra. Hence, $P_1$ and $P_2$ may be included in PL collapsible polyhedra $Q_1$ and $Q_2$ that cover $K_{\mathcal P}$.
\end{proof}

The presentation complex of a (homogeneous) one-relator presentation is an inner-connected polyhedron and each of its points has a connected link, except possibly the wedge point. Moreover, if such a complex admits no algebraic collapses only the wedge point may be splittable. It is a consequence of Theorem \ref{No2Theo} that most such complexes do not admit PL collapsible covers of size two.

\begin{prop}\label{ExcepProp}
Let ${\mathcal P} = \langle x_1, \dots, x_k \, \vert \, r \rangle$ be a finite one-relator presentation such that $r$ does not admit algebraic collapses. Then $\plgcat(K_{ {\mathcal P} }) = 2$ if and only ${\mathcal P}$ is of the form $\langle \, x \, |  \, (xx^{-1})^{\pm 1} \, \rangle$.
\end{prop}

\begin{proof}
By cellular homology, the group $H_2(K_{ {\mathcal P} })$ is free abelian of rank at most 1. Moreover, by a straightforward Euler characteristic computation we know that
\[ 
\rk H_2(K_{ {\mathcal P} }) - \rk H_1( K_{ {\mathcal P} } ) = 1 - k.
\]
Hence, if ${\mathcal P}$ has $k > 1$ generators, we have $\rk H_1(K_{ {\mathcal P} }) < \rk H_2(K_{ {\mathcal P} })$ and the conclusion follows from Theorem \ref{No2Theo}. The case $H_2(P) \equiv 0$ is also covered by Theorem \ref{No2Theo}.
\par It remains then to analyze the case of one-relator presentations with one generator and non trivial second homology group. Those are exactly the presentations of the form $\langle \, x \,| \, r \, \rangle$, where $r$ is a word on letters $x$, $x^{-1}$ with total exponent $0$. Suppose that a triangulation of $K_{\mathcal P}$ admits a cover by collapsible subcomplexes $K_1$, $K_2$. We may assume that $K_1 \cap K_2$ is a graph with at most one leaf. Since $H_0(K_1 \cap K_2) \equiv \ZZ^{2}$ and $H_1(K_1 \cap K_2) \equiv \ZZ$, one of the connected components of $K_1 \cap K_2$ is acyclic and therefore consists of only one point. For this to be possible, the link of the wedge point must have more than one connected component. By drawing the Whitehead graph of $r$ (see \cite[Ch.6]{HAM}), we see that this is the case only for presentations of the form $\langle \, x \, | \, (xx^{-1})^{\pm n} \rangle$, $n \in \NN$. Call $C$ the other connected component of $K_1 \cap K_2$. Since it is a connected graph with one cycle and no leaves, $C$ is homeomorphic to $S^{1}$. Moreover, it is not difficult to show that the intersection of this component with the loop $x$ may by assumed to be 0-dimensional, that is, a finite set of points. Also, notice that the intersection $C \cap x$ is not empty. Indeed, suppose that the loop $x$ is entirely contained in $K_1$ (the argument for $K_2$ is identical). Since the homology class determined by $x$ is a generator of the first homology group $H_1(K_{ {\mathcal P} })$, $K_1$ does not have trivial $H_1$, a contradiction. Let then $v$ be a point in $C \cap x$ and let $a$, $b$ the edges of the subdivision of $x$ that contain $v$. Since the intersection of $C$ with the loop $x$ is 0-dimensional, we may assume that $a \in K_1 \setminus K_2$ and $b \in K_2 \setminus K_1$. The edges $a$, $b$ are faces of $2n$ 2-simplices in any triangulation of $K_{ {\mathcal P} }$. Furthermore, the (open) star of $v$ is homeomorphic to a union of $2n$ half euclidean planes with the $x$ axis identified. It follows that vertex $v$ has valency $2n$ in the graph $C$. This is impossible unless $n = 1$. Finally, observe that the complex associated to a presentation of the form $\langle \, x \, |  \, (xx^{-1})^{\pm 1} \, \rangle$ is homeomorphic to a 2-sphere with its poles identified and so admits a cover by two PL collapsible subpolyhedra.
\end{proof}

As a corollary to Propositions \ref{StripProp} and \ref{ExcepProp} we obtain the announced characterization, which shows that the property of admitting a PL collapsible cover of size two is very restrictive for this class.

\begin{theo}\label{plgcatonerel}
Let ${\mathcal P} = \langle x_1, \dots, x_k \, \vert \, r \rangle$ be a finite one-relator presentation. Then $K_{ {\mathcal P} }$ can covered by two PL collapsible subpolyhedra if and only if $r$ admits an algebraic collapse or ${\mathcal P}$ is of the form $\langle \, x \, |  \, (xx^{-1})^{\pm 1} \, \rangle$.
\end{theo}

\section{Inner-connected polyhedra}

In the previous sections the class of inner-connected polyhedra is used as a tool to give a local-global criterion to estimate the PL geometric category. We now delve a little deeper into the structure of these polyhedra. Specifically, we recover for this class a version of the following well-known result: a closed surface can be obtained by making identifications on pairs of boundary edges of an appropriate polygon. We follow the treatment and notation of \cite[Ch.6]{Lee}.

\begin{defi}
Given a finite alphabet $S$, a \textit{word} of length $k$ in $S$ is an ordered list of $k$ symbols of $S \cup S^{-1}$. A \textit{polygonal presentation} $\mathfrak{P}$ is a finite alphabet $S$ together with a finite set of words $W_1, \dots, W_r$ in $S$ of length at least three such that every element of $S$ (or its formal inverse) appears in some word. We denote such a presentation by $\mathfrak{P} = \langle S \, | \, W_1, \dots , W_r \rangle$.
\end{defi}

A polygonal presentation $\mathfrak{P}$ determines a topological space (called the \textit{geometric realization} of $\mathfrak{P}$) in the following fashion. For each word $W_i$ in $\mathfrak{P}$ of length $k$ form the regular convex polygon of $k$	sides $P_i$ and label its edges in counterclockwise order according to $W_i$, starting by an arbitrary vertex. Now identify edges with the same label in $\coprod_{i} P_i$ by the simplicial homeomorphism that matches the vertices of the edges, inverting orientation when necessary.
\par There is a number of combinatorial movements on polygonal presentations, called \textit{elementary transformations}, that preserve the (PL) homeomorphism type of the corresponding geometric realizations. We describe here only the transformations we will use and refer to \cite[Ch.6]{Lee} for the complete list.

\begin{itemize}
	\item Reflection: $\langle S \, | \, a_1 \dots a_m, W_2, \dots, W_r \rangle \mapsto \langle S \, | \, a_m^{-1} \dots a_1^{-1}, W_2, \dots, W_r \rangle$.
	\item Rotation: $\langle S \, | \, a_1 a_2 \dots a_m, W_2, \dots, W_r \rangle \mapsto \langle S \, | \, a_2 \dots a_m a_1, W_2, \dots, W_r \rangle$.
	\item Pasting: $\langle S,e \, | \, W_1e , e^{-1}W_2, \dots, W_r \rangle \mapsto \langle S \, | \, W_1W_2, \dots, W_r \rangle$. Note that $e$ does not belong to $S$ so that none of the words $W_1 ,\dots, W_r$ should contain $e$ for this transformation to be valid.
\end{itemize}

The main result of this section states that every inner-connected polyhedron $P$ has a polygonal presentation with one word. The strategy of the proof consists of repeatedly pasting pairs of 2-simplices of a triangulation of $P$ joined by an inner edge until we are left with only one polygon. However, the vertices of an inner edge may be \textit{singular}, i.e. they may have a neighborhood that is not homeomorphic to an open disk. By pasting a pair of 2-simplices through an inner edge with singular vertices, we may create identifications in the interior of the polygon we are building. We show in the next lemma that we can avoid this situation by considering a sufficiently fine triangulation of $P$.

\begin{lemma}\label{regvert}
Let $K$ be an inner-connected simplicial complex of dimension $2$. Then, each pair of simplices of the second barycentric subdivision $K''$ of $K$ may be joined by an inner sequence such that the vertices of the inner edges in the sequence are not singular.
\end{lemma}

\begin{proof}
Let $\sigma_1$, $\sigma_2$ be a pair of 2-simplices of $K''$ and let $\tau_1$, $\tau_2$ be respectively the 2-simplices of $K$ containing them. Since $K$ is inner-connected, there is an inner sequence $S$ in $K$ joining $\tau_1$ to $\tau_2$. It is not difficult to find an inner sequence in $K''$ formed by 2-simplices contained in 2-simplices of $S$ that avoids the vertices of $K$ and joins $\sigma_1$ to $\sigma_2$.
\end{proof}

\begin{theo}\label{onepoly}
Let $P$ be an inner-connected polyhedron. Then $P$ admits a polygonal presentation with only one word.
\end{theo}

\begin{proof}
Take a simplicial complex $K$ that triangulates $P$. By Lemma \ref{regvert}, we may assume that every pair of 2-simplices of $K$ is joined by an inner sequence such that the inner edges involved do not have singular vertices. Choose a different label for each edge of $K$ and fix an orientation for its simplices. Consider the polygonal presentation $\mathfrak{P}$ that has as alphabet the set of labels of edges of $K$ and a word for each 2-simplex, determined by the edges of its boundary in the order given by the prescribed orientations. Since the geometric realization of $\mathfrak{P}$ is homeomorphic to $P$, it suffices to reduce $\mathfrak{P}$ to a presentation with one word by applying elementary transformations. The 2-simplex that corresponds to $W_1$ has at least one inner edge $a$ with no singular vertices. Without loss of generality, assume that $W_2$ is the only other word in which $a$ or $a^{-1}$ appears. By applying rotations and reflections we may assume that $W_1 = \tilde{W}_1 a$, $W_2 = a^{-1}\tilde{W}_2$ and paste them to reduce the number of words in $\mathfrak{P}$. Inductively, suppose that there is more than one word in the presentation. We claim that there is one inner edge with no singular vertices of $K$ that appears exactly once in $W_1$. Indeed, if it was not the case, it would be impossible to connect a 2-simplex of (the subcomplex determined by) $W_1$ and a 2-simplex not in $W_1$ by an inner sequence with no singular vertices in its edges. As before, rearrange the words and perform rotations and reflections in such a way that it is possible to paste words $W_1$ and $W_2$.
\end{proof}

\begin{rem}
Let $P$ be an inner-connected polyhedron and let $\mathfrak{P} = \langle S \, |\, W \rangle$ be a polygonal presentation of $P$ with one word obtained as in Theorem \ref{onepoly}. Consider the subgraph $G$ of $P$ formed by the edges determined by $S$. The word $W$ defines a surjective combinatorial map $\varphi: S^{1} \to G$ for a suitable triangulation of $S^{1}$. This provides an alternative description of inner-connected polyhedra. Concretely, given a simplicial graph $G$ and a surjective combinatorial map $\varphi: S^{1} \to G$ we obtain an inner-connected polyhedron as the space underlying the \rm{CW}-complex that consists of one 2-cell attached to $G$ according to $\varphi$. In the case that $G$ is homeomorphic to a bouquet of spheres of dimension 1, the resulting space is a (homogeneous) one-relator presentation complex. As another example, if the combinatorial map $\varphi$ touches every edge of $G$ at most twice, the inner-connected polyhedron we get from this construction is a pseudosurface.
\end{rem}

\section*{Acknowledgments}
The author is grateful to Gabriel Minian for many useful discussions during the preparation of this article.

\end{document}